\documentclass[12pt]{amsart}

\usepackage[T1]{fontenc}
\usepackage{lmodern}
\usepackage{amsmath,amssymb,amsthm,mathtools}
\usepackage{microtype}
\usepackage[hidelinks]{hyperref}

\usepackage{needspace}

\numberwithin{equation}{section}
\newtheorem{theorem}{Theorem}[section]
\newtheorem{lemma}[theorem]{Lemma}
\newtheorem{proposition}[theorem]{Proposition}

\theoremstyle{definition}

\theoremstyle{remark}
\newtheorem{remark}[theorem]{Remark}

\numberwithin{equation}{section}
\newcommand{\NN}{\mathbb N}
\newcommand{\ZZ}{\mathbb Z}
\newcommand{\RR}{\mathbb R}
\newcommand{\Ass}{\operatorname{Ass}}
\newcommand{\astab}{\operatorname{astab}}
\newcommand{\supp}{\operatorname{supp}}
\newcommand{\cost}{\operatorname{cost}}
\newcommand{\depth}{\operatorname{depth}}
\newcommand{\NP}{\operatorname{NP}}

\DeclareMathOperator{\conv}{conv}

\newcommand{\kk}{\Bbbk}
\newcommand{\Rees}{\mathcal R}
\newcommand{\mm}{\mathfrak m}
\newcommand{\yd}{\deg_Y}

\numberwithin{equation}{section}

\begin{document}


\title{Polynomial bounds for the stability of associated primes of monomial ideals}

\author{Tran Nam Trung}
\address{Institute for Artificial Intelligence, Vnu University of Engineering and Technology, Hanoi, Vietnam}
\email{tntrung@vnu.edu.vn}

\dedicatory{Dedicated to Professor Le Tuan Hoa on the occasion of his 70th birthday}


\begin{abstract}
Let $I$ be a nonzero proper monomial ideal in
$\kk[x_1,\ldots,x_r]$, and let $d$ be the maximal degree of a
minimal monomial generator. We prove that the index of stability
of the associated primes of its ordinary powers satisfies
\[
 \astab(I)\le\frac{(4e)^{r-1}}4\,d^{r-1}\qquad(r\ge3),
\]
while $\astab(I)=1$ for $r=2$. If $I$ has $s\ge2$ minimal
generators, we also obtain
\[
 \astab(I)\le\min\{r,s\}(s-1)d^{\min\{r-1,s-2\}}.
\]
The first estimate answers Hoa's question on polynomial bounds.
The proofs reduce stabilization to integer optimization for a
matrix whose last row records the power of the ideal. We bound
the degrees of its circuits and Graver basis elements and prove
proximity estimates for optimal solutions in fixed fibers.
A Steinitz rearrangement and a lattice-point count in the
difference body of a simplex give the generator-independent bound.
The matrix estimates hold for arbitrary indexed integral
configurations in a simplex and real linear objective functions.
The stability bounds hold over an arbitrary field, without
normality or equigeneration assumptions.
\end{abstract}

\maketitle


\section{Introduction}\label{sec:introduction}

Integer matrices connect the algebra of monomial ideals with
integer programming and the geometry of lattice points. A product
of monomial generators is encoded by a nonnegative integer vector,
and equality of exponent vectors is encoded by an integer kernel.
Circuits and Graver bases describe the corresponding relations.
In this paper we use these matrix structures to bound the power
from which the associated primes of an arbitrary monomial ideal
remain constant.

Let $R=\kk[x_1,\ldots,x_r]$ be a polynomial ring over a field,
where $r\ge2$, and let $I\subsetneq R$ be a nonzero monomial ideal.
Brodmann~\cite{BrodmannAss} proved that $\Ass_R(R/I^n)$ is constant
for all sufficiently large $n$. The \emph{index of stability} is
\[
 \astab(I)=\min\{N\ge1:\Ass_R(R/I^n)=\Ass_R(R/I^N)
                              \text{ for every }n\ge N\}.
\]
Write $G(I)$ for the minimal monomial generating set, and put
\[
 d=d(I)=\max\{\deg u:u\in G(I)\},\qquad s=\mu(I)=|G(I)|.
\]
Our main aim is a bound polynomial in $d$ for fixed $r$ that is
independent of the number of generators.

Hoa~\cite{Hoa} obtained an explicit general bound using systems
of integer inequalities. Its exponent in $d$ depends on $s$;
substituting a bound for $s$ in terms of $r$ and $d$ does not give
a polynomial bound for fixed $r$. In his survey, Hoa asked whether
$\astab(I)$ admits a bound of order $d(I)^r$;
see~\cite[Question~4(i)]{HoaSurvey}. Our first result answers this
question with exponent $r-1$.

\begin{theorem}\label{thm:main}
Let $I$ be a nonzero proper monomial ideal in
$R=\kk[x_1,\ldots,x_r]$, and let $d=d(I)$. If $r=2$, then
$\astab(I)=1$. If $r\ge3$, then
\begin{equation}\label{eq:main}
 \astab(I)\le \frac{(4e)^{r-1}}4\,d^{r-1}.
\end{equation}
\end{theorem}

The coefficient in~\eqref{eq:main} depends only on the number of
variables. In particular, the bound remains polynomial in $d$ when
the number of minimal generators varies. For ideals with few
generators, a different estimate can be substantially smaller.

\begin{theorem}\label{thm:generators}
Let $I\subsetneq R=\kk[x_1,\ldots,x_r]$ be a nonzero monomial ideal
with $d=d(I)$ and $s=\mu(I)$. If $s=1$, then $\astab(I)=1$.
If $s\ge2$, then
\begin{equation}\label{eq:generator-main}
 \astab(I)\le
 \min\{r,s\}(s-1)d^{\min\{r-1,s-2\}}.
\end{equation}
Moreover, $\astab(I)=1$ when $s=2$.
\end{theorem}

The associated primes of ordinary powers can both appear and
disappear before stabilization. Heuberger, Rath, and Rissner
\cite{HRR} developed a general method for bounding the
\emph{copersistence index}, the least power from which
\[
 \Ass_R(R/I^{n+1})\subseteq\Ass_R(R/I^n)
\]
holds for every subsequent $n$. Their construction separates the
choice of an integer inequality system from the estimate for its
solutions, and improves Hoa's copersistence bound by an exponential
factor. Theorems~\ref{thm:main} and~\ref{thm:generators} give bounds
for full stabilization, including the powers at which an associated
prime can disappear. Recent examples of fluctuations in the
occurrence of the maximal ideal~\cite{CNQ} further illustrate the
need for this additional control.

The matrix used here has the form
\[
 D=\begin{pmatrix}I_m&p_1&\cdots&p_s\\0&1&\cdots&1\end{pmatrix},
 \qquad p_j\in d\Delta_m\cap\ZZ^m,\qquad m=r-1,
\]
where $\Delta_m$ is the standard simplex. For an arbitrary real
linear objective, we compare its optimal solutions in a nonempty
integer fiber with those of the linear relaxation. For $s\ge2$,
put $\Gamma=(s-1)d^{\min\{m,s-2\}}$. Our circuit argument bounds
the $Y$-degree of every Graver basis element by $\Gamma$ and
gives an integral optimum $u$ near any real optimum $z$ with
\[
 \frac12\|u_Y-z_Y\|_1<\Gamma.
\]
These estimates are established for arbitrary indexed configurations
$p_1,\ldots,p_s$, including repeated points. The proofs use
determinants, the normalized volumes of simplices, and
sign-compatible circuit decompositions.

An alternative Graver degree bound is independent of $s$.
The last row of $D$ permits positive and negative generator
occurrences to be paired before the Steinitz lemma is applied.
The resulting zero-sum vectors lie in
$d(\Delta_m-\Delta_m)$, so the count takes place in dimension
$m=r-1$. The Ehrhart polynomial of the type-$A_m$ root polytope
then bounds the number of partial sums. Proximity estimates are
obtained by applying the same rearrangement to a sequence with
auxiliary vectors; the contribution of fractional slack
coordinates is included explicitly. For an optimal vertex this
also gives an integral optimum with $\|u_Y-z_Y\|_1\le2mQ$,
where $Q$ is the lattice-point count and is independent of $s$.

The algebraic application uses the graded $\mathfrak m$-socle
of $\mathcal R(I)/x_1\mathcal R(I)$, where
$\mathfrak m=(x_1,\ldots,x_r)$. Its degree-$n$ component detects
the occurrence of $\mathfrak m$ in $\Ass_R(R/I^n)$.
The matrix bounds control its generating degrees and its last
nonzero degree when it is finite-dimensional. Compact faces of
the Newton polyhedron supply optimal solutions with support
bounded by the analytic spread. Monomial localization then gives
one stabilization threshold for all associated primes.

For integral closures of monomial powers, bounds of order
$d^{r-2}$ are known from Trung~\cite{Trung} and
Hoa--Trung~\cite{HoaTrung}. They apply to ordinary powers of
normal monomial ideals. The normal ideals
$J=\overline{K^r}$ constructed in~\cite{Hoa} have stability
indices bounded below by a positive multiple of $d(J)^{r-2}$
for fixed $r\ge4$. In relating the successive quotients used
there to ordinary powers, one uses
$\Ass_R(R/J^n)=\Ass_R(J^{n-1}/J^n)$; see~\cite{Trung}.
Whether the exponent $r-1$ in Theorem~\ref{thm:main} can be
reduced to $r-2$ for arbitrary monomial ideals is not resolved
by the estimates established here.

Section~\ref{sec:matrix} proves the circuit, Graver degree, and
proximity estimates for integer matrices. Section~\ref{sec:algebra}
establishes their algebraic application and proves
Theorem~\ref{thm:generators}. Section~\ref{sec:lattice} obtains
the generator-independent stabilization bound. Section
\ref{sec:uniform} combines the estimates to prove
Theorem~\ref{thm:main}.

\section{Integer matrices, circuits, and proximity}\label{sec:matrix}

Let $m,s,d$ be positive integers with $s\ge2$. Choose an indexed
configuration $p_1,\ldots,p_s\in\NN^m$ with $|p_j|\le d$.
The points are allowed to repeat. Write $[q]=\{1,\ldots,q\}$
for a positive integer $q$. For each positive integer $q$,
write $\Delta_q=\{v\in\RR_{\ge0}^q:\sum_i v_i\le1\}$.
Set $\Delta_0=\{0\}$. Throughout, $\NN$ includes zero,
and $|p|$ denotes the sum of the coordinates of $p\in\NN^m$.
Consider
\begin{equation}\label{eq:D}
 D=\begin{pmatrix}
 e_1&\cdots&e_m&p_1&\cdots&p_s\\
 0&\cdots&0&1&\cdots&1
 \end{pmatrix}.
\end{equation}
Its first $m$ columns are slack columns, labelled $X$, and the
remaining columns are labelled $Y_1,\ldots,Y_s$. For a vector
$u=(\kappa,\beta)$, write $u_Y=\beta$ and
$\yd u=\sum_j\beta_j$. The homogenizing row gives
\begin{equation}\label{eq:rs-kernel-dimension}
 \operatorname{rank}D=m+1,\qquad
 \dim_{\RR}\ker_{\RR}D=s-1.
\end{equation}

For $q=(b,n)\in\NN^m\times\NN$, define
\[
 F_q=\{u\in\NN^{m+s}:Du=q\},\qquad
 P_q=\{u\in\RR_{\ge0}^{m+s}:Du=q\}.
\]
Every nonempty $P_q$ is bounded: the $Y$-coordinates sum to $n$,
and the slack coordinates are bounded by $b$. An arbitrary real
linear cost $\cost(u)=c\cdot u$ therefore has an optimal vertex
on $P_q$. Every vertex has at most $m+1$ positive coordinates,
because the columns indexed by its positive coordinates are
linearly independent.

For real vectors $h,w$, write $h\sqsubseteq w$ when they lie
in the same closed orthant and $|h_i|\le|w_i|$ for every $i$.
Thus $h^+\le w^+$ and $h^-\le w^-$. The \emph{Graver basis}
$\mathcal G(D)$ consists of the nonzero integer kernel vectors
minimal under $\sqsubseteq$. Every integer kernel vector is a
conformal sum of Graver basis elements: subtracting a proper
conformal subrelation decreases its integer $\ell_1$-norm.

Set
\begin{equation}\label{eq:rs-parameters}
 \gamma=\min\{m,s-2\},\qquad \Gamma=(s-1)d^\gamma.
\end{equation}
A \emph{circuit} of an integer matrix is a nonzero integer kernel vector
with inclusion-minimal support, normalized so that the greatest common
divisor of its coordinates is one.  An \emph{affine circuit} is a minimally
affinely dependent indexed configuration of points.  Two real vectors are
\emph{sign-compatible} if they lie in one closed coordinate orthant.
When decomposing a vector, we require in addition that every summand
vanish outside its support.  Sign compatibility alone does not impose
the coordinatewise bounds in the definition of conformality used above.
For background on triangulations and Graver bases, see
\cite{DRS,Onn}.  The estimates needed here are proved below.

\subsection{Circuit degrees}

\begin{lemma}[Affine circuit degree]\label{lem:rs-affine}
Let $t\ge1$, and let $p_1,\ldots,p_{t+2}\in\NN^t$ form an affine
circuit whose affine hull is $\RR^t$.  Assume $|p_i|\le d$ for every $i$.
If $c\in\ZZ^{t+2}$ is its primitive affine relation, then
\[
 \sum_i c_i^+=\sum_i c_i^-\le d^t.
\]
For $t=0$, the same assertion holds for the indexed configuration of two
equal points: its primitive relation has degree one.
\end{lemma}

\begin{proof}
For $t\ge1$, let $B$ be the matrix with columns $(p_i,1)$.
It has rank $t+1$, and its signed maximal minors form a nonzero
integer kernel vector $b$.  All these minors are nonzero, since
every proper subconfiguration is affinely independent.  After changing
signs if necessary, write $c=b/g$, where $g\ge1$ is the greatest
common divisor of the coordinates of $b$.

Put $P=\conv\{p_1,\ldots,p_{t+2}\}$.  The simplices
\[
 P_i=\conv\{p_j:j\ne i\},\qquad b_i>0,
\]
cover $P$ and have pairwise disjoint interiors.  Indeed, given a convex
representation $x=\sum_i\lambda_i p_i$, move the coefficient vector
in the direction $-b$ until the first coefficient with $b_i>0$ becomes
zero.  During this motion all coefficients remain nonnegative and their
sum remains one.  The resulting representation lies in one of the $P_i$.
Moreover, all affine representations of $x$ differ by a multiple of $b$.
The endpoint just obtained is consequently independent of the initial
representation.  If two indices vanish there, the point belongs to the
common face obtained by omitting both indices.  This proves the assertion
about interiors.

The determinant $|b_i|$ is $t!$ times the Euclidean volume of $P_i$.
Thus
\[
 \sum_{b_i>0}|b_i|=t!\operatorname{vol}_t(P).
\]
Since $P\subseteq d\Delta_t$ and
$t!\operatorname{vol}_t(d\Delta_t)=d^t$, we get
\[
 \sum_i c_i^+=\frac{t!\operatorname{vol}_t(P)}{g}\le d^t.
\]
The equality of the positive and negative sums follows from the affine
relation.  For $t=0$, the primitive relation is $(1,-1)$, up to sign.
\end{proof}

\begin{lemma}[Circuit degrees]\label{lem:rs-circuit}
Let $c$ be a circuit of $D$, and let $k$ be the number of its nonzero
$Y$-coordinates.  Then $k\ge2$ and
\begin{equation}\label{eq:rs-circuit}
 \yd(c^+)=\yd(c^-)\le d^{k-2}\le d^\gamma.
\end{equation}
\end{lemma}

\begin{proof}
The homogenizing row gives $\sum_j c_{Y_j}=0$.  Since the slack columns
are linearly independent, a nonzero relation uses $Y$-columns, and
there must be at least two of them.

Let $U\subseteq[m]$ index the slack columns in
$\supp(c)$.  Project the $k$ vectors $p_j$ occurring in this
support onto the coordinates not indexed by $U$.
Every affine relation among the projected points lifts uniquely to a
relation on these $Y$-columns and the slack columns in $U$: the
coefficients of the latter columns are determined by the omitted
coordinate equations.  The kernel on the support of a circuit is
one-dimensional.  Otherwise, a linear combination of two independent
relations could be chosen to cancel a nonzero coordinate, producing a
relation with smaller support.  It follows that the projected points
have a one-dimensional affine-relation space.  No proper subset is
affinely dependent, since such a dependence would lift to a relation
with strictly smaller support.  They therefore form an affine circuit
of dimension $k-2$.

Choose $k-2$ of the retained coordinate functions that are linearly
independent on its affine difference space.  Projection to these
coordinates is injective on that difference space and preserves all
affine relations.  The resulting points $p_1,\ldots,p_k$ lie in
$d\Delta_{k-2}$, since projection preserves nonnegativity and does
not increase total coordinate sum.

The vector $c_Y$ is primitive as an integer vector.  Indeed, any common
divisor of its coordinates divides every slack coefficient, those
coefficients being integer linear combinations of $c_Y$; primitivity
of $c$ then forces that divisor to be one.
Lemma~\ref{lem:rs-affine} gives $\yd(c^+)\le d^{k-2}$.
Finally, a circuit of a rank-$(m+1)$ matrix has at most $m+2$ columns in its
support, and $k\le s$.  Thus $k-2\le\gamma$, as required.
\end{proof}

\subsection{Sign-compatible decompositions and proximity}

\begin{lemma}[Circuit decomposition]\label{lem:rs-decomposition}
Every nonzero $w\in\ker_{\RR}D$ has a representation
\begin{equation}\label{eq:rs-decomposition}
 w=\sum_{\nu=1}^{p}\lambda_\nu c^{(\nu)},\qquad
 p\le s-1,\qquad \lambda_\nu>0,
\end{equation}
where the $c^{(\nu)}$ are primitive circuits, sign-compatible with $w$
and supported in $\supp(w)$.
\end{lemma}

\begin{proof}
Intersect $\ker_{\RR}D$ with the closed orthant determined by the signs
of $w$, and set all coordinates outside $\supp(w)$ equal to zero.
The resulting cone is a pointed rational polyhedral cone.  Its extreme
rays are generated by circuits.  To verify this, a nonzero vector $v$
on such a ray must have a one-dimensional kernel on its support: if
that kernel had dimension at least two, sufficiently small positive and
negative perturbations by an independent kernel vector would stay in
the same orthant and would contradict extremality.  Hence its support
is minimal.  Rationality permits a primitive integer generator of the
ray.  Conversely, a circuit in the cone generates an extreme ray,
because any decomposition in the cone is supported on its support,
where the kernel is one-dimensional.

The cone is generated by its extreme rays.  Conic Carath\'eodory,
applied in $\ker_{\RR}D$, gives a positive combination of at most
$\dim\ker_{\RR}D=s-1$ ray generators.  This proves the lemma.
\end{proof}

\begin{lemma}[Bound for the $Y$-degrees of Graver basis elements]\label{lem:rs-graver}
Every $h\in\mathcal G(D)$ satisfies
\begin{equation}\label{eq:rs-graver}
 \yd(h^+)=\yd(h^-)\le \Gamma.
\end{equation}
\end{lemma}

\begin{proof}
Apply Lemma~\ref{lem:rs-decomposition} to $h$.
If some $\lambda_\nu\ge1$, sign compatibility gives
\[
 (c^{(\nu)})^+\le h^+,\qquad
 (c^{(\nu)})^-\le h^-.
\]
Minimality under $\sqsubseteq$ forces $h=c^{(\nu)}$, and the conclusion
follows from Lemma~\ref{lem:rs-circuit}.
Otherwise all coefficients are less than one.  Since no coordinates
cancel in the decomposition,
\[
 \yd(h^+)
 =\sum_{\nu=1}^{p}\lambda_\nu\yd((c^{(\nu)})^+)
 <p d^\gamma\le(s-1)d^\gamma=\Gamma.
\]
The two $Y$-degrees agree by the homogenizing row of $D$.
\end{proof}

\begin{lemma}[Proximity to the linear relaxation]\label{lem:rs-proximity}
Assume $F_q\ne\varnothing$ and $n\ge1$.
For every real linear cost and every cost-minimizing point
$z\in P_q$, there is an integral cost minimum $u\in F_q$ such that
\begin{equation}\label{eq:rs-proximity-norm}
 \frac12\|u_Y-z_Y\|_1
 =\sum_j(u_{Y_j}-z_{Y_j})^+<\Gamma.
\end{equation}
Moreover,
\begin{equation}\label{eq:rs-proximity}
 \sum_{j\notin\supp(z_Y)}u_{Y_j}\le \Gamma-1.
\end{equation}
\end{lemma}

\begin{proof}
The integral fiber is finite.  Among its cost minima choose $u$
minimizing $\|u-z\|_1$, and put $w=u-z\in\ker_{\RR}D$.
There is no nonzero integer kernel vector $h$ conformal to $w$.
Indeed,
\[
 h^+\le w^+\le u,\qquad h^-\le w^-\le z,
\]
so $u-h\in F_q$ and $z+h\in P_q$.
Optimality of $u$ and $z$ gives, respectively,
$\cost(h)\le0$ and $\cost(h)\ge0$.
Consequently $u-h$ is another integral cost minimum, whereas
\[
 \|u-h-z\|_1=\|u-z\|_1-\|h\|_1<\|u-z\|_1,
\]
a contradiction.

If $w\ne0$, apply Lemma~\ref{lem:rs-decomposition} to $w$.
Every coefficient in the resulting decomposition is strictly less
than one; otherwise a circuit is an excluded integer kernel vector
conformal to $w$.  Lemma~\ref{lem:rs-circuit} now gives
\[
 \yd(w^+)<(s-1)d^\gamma=\Gamma.
\]
The last row of $Dw=0$ gives
$\yd(w^+)=\yd(w^-)=\frac12\|u_Y-z_Y\|_1$,
which proves~\eqref{eq:rs-proximity-norm}.
For $j\notin\supp(z_Y)$, we have $u_{Y_j}=w_{Y_j}\ge0$.
Thus the sum in~\eqref{eq:rs-proximity} is an integer strictly less
than $\Gamma$, proving the bound.  If $w=0$, both assertions are immediate.
\end{proof}

\begin{remark}\label{rem:rs-proximity}
This estimate uses $u-z$, which lies in the real kernel of $D$,
rather than $u-\lfloor z\rfloor$.  It therefore requires no correction
for fractional slack coordinates.  It bounds the total number of
$Y$-occurrences outside the support of $z_Y$, not the support size
of an integral optimum.  The factor $s-1$ enters through
\eqref{eq:rs-kernel-dimension}.
\end{remark}

\subsection{A column-independent Graver degree bound}

We also bound Graver degrees independently of the number of
columns. Let
\[
 \sigma_m=m-1+\frac1m,\qquad
 B_\triangle=\Delta_m-\Delta_m.
\]
The origin-symmetric convex body $B_\triangle$ is the unit ball
of the norm
\begin{equation}\label{eq:triangle-norm}
 \|v\|_\triangle=\max\{\textstyle\sum_i v_i^+,
                                  \textstyle\sum_i v_i^-\}.
\end{equation}
This norm takes integer values on $\ZZ^m$.

\begin{lemma}[Steinitz rearrangement]\label{lem:Steinitz}
Let $B\subset\RR^m$ be an origin-symmetric convex body. If
$v_1,\ldots,v_\nu\in B$ and $\sum_jv_j=0$, there is an ordering
whose partial sums belong to $\sigma_mB$.
\end{lemma}

\begin{proof}
This is~\cite[Corollary~2, p.~67]{Sev}. Equivalently, apply the
compact-summation theorem there with $a=0$ and
$B_0=\conv\{v_1,\ldots,v_\nu\}$. It places the partial sums in
\[
 (m-1)B_0+\conv\bigl(\{0\}\cup(-B_0/m)\bigr)
 \subseteq\left(m-1+\frac1m\right)B.
\]
For the classical constant $m$ and its use in integer programming,
see~\cite[Section~1.2]{EW}.
\end{proof}
\begin{lemma}[Cyclic counting]\label{lem:cycles}
Suppose a finite multiset of integer vectors has no
nonempty zero-sum submultiset. Add $k\ge1$ auxiliary integer vectors
so that the total sum is zero. If the combined vectors admit an
ordering whose partial sums belong to one set of $Q$ lattice
points, then the number of vectors in the original multiset is at
most $k(Q-1)$.
\end{lemma}

\begin{proof}
Identify the final partial-sum position with the initial one.
The $k$ auxiliary edges divide the resulting cycle into $k$,
possibly empty, arcs containing no auxiliary edges. An arc of $a$
edges has
$a+1$ distinct partial-sum labels: equal labels would give a
nonempty zero-sum submultiset of vectors from the original multiset.
Every label
belongs to the same set of $Q$ lattice points, so $a\le Q-1$.
For the arc passing through the final and initial positions,
the same labels are used: the final partial sum equals the initial
one because the total sum is zero. Thus equality of two labels on
this arc also gives a zero-sum submultiset. Summing over the arcs
gives the result.
Zero auxiliary vectors are allowed.
\end{proof}

\begin{lemma}[Decomposition of integer vectors]\label{lem:signedmass}
Let $d,k\ge1$ be integers.
If $v\in\ZZ^m$ satisfies
\[
 \sum_i v_i^+\le kd,\qquad \sum_i v_i^-\le kd,
\]
then $v=v^{(1)}+\cdots+v^{(k)}$ for integer vectors $v^{(j)}$
with $\|v^{(j)}\|_\triangle\le d$ for every $j$.
\end{lemma}

\begin{proof}
Represent $v^+$ and $v^-$ separately as multisets of coordinate
unit vectors. Partition each multiset into $k$ labelled bins,
each containing at most $d$ units; empty bins are allowed.
For each label, subtract the sum of the negative units from the
sum of the positive units in that bin.
The resulting integer vector has both sign sums at most $d$.
These $k$ vectors sum to $v$.
\end{proof}

For an integer $\Lambda\ge0$, define
\begin{equation}\label{eq:Phi}
 \Phi_m(\Lambda)=
 \sum_{\substack{p,q\ge0\\p+q\le m}}
 \frac{m!}{p!q!(m-p-q)!}
 \binom{\Lambda}{p}\binom{\Lambda}{q},
\end{equation}
with $\binom{\Lambda}{p}=0$ when $p>\Lambda$. Choosing the
disjoint supports of the positive and negative coordinates shows
that this is the number of lattice points in $\Lambda B_\triangle$:
positive entries on a fixed support of size $p$ with sum at most
$\Lambda$ can be chosen in $\binom{\Lambda}{p}$ ways, and the
negative entries are counted similarly. Put
\begin{equation}\label{eq:lattice-parameters}
 \Lambda=\lfloor\sigma_md\rfloor,\qquad Q=\Phi_m(\Lambda).
\end{equation}

\begin{lemma}\label{lem:primitive}
Every $h\in\mathcal G(D)$ satisfies
\[
 \yd(h^+)=\yd(h^-)\le Q.
\]
\end{lemma}

\begin{proof}
The equality follows from the last row of $Dh=0$. Pair every
positive $Y$-occurrence with a negative one, recording a pair of
types $j,l$ as the vector $p_j-p_l$. Record each signed slack
occurrence as a signed unit vector. These integer vectors sum to
zero and have $\triangle$-norm at most $d$, because each $p_j$
has nonnegative coordinates with sum at most $d$.

A nonempty proper zero-sum submultiset would restore to a nonzero
proper integer kernel vector conformal to $h$. No cancellation
occurs on restoration, since the positive and negative coordinate
supports of $h$ are disjoint. Such a subrelation contradicts
Graver minimality. Lemma~\ref{lem:Steinitz} orders the vectors
with partial sums in $\sigma_mdB_\triangle$. Their norms are
integers, so all partial sums lie in $\Lambda B_\triangle$.
The partial sums preceding the final return to zero are distinct;
otherwise there would be a proper zero-sum submultiset. The number
of vectors, and therefore the number of paired $Y$-occurrences,
is at most $Q$.
\end{proof}

\subsection{Proximity independent of the number of columns}

\begin{proposition}\label{prop:matrix-proximity}
Assume $F_q\ne\varnothing$ and $n\ge1$, and let $z$ be a cost-minimizing
vertex of $P_q$ for an arbitrary real linear cost. Put
\[
 t=|\supp(z_Y)|,\qquad
 d_0=n-\sum_j\lfloor z_{Y_j}\rfloor,
\]
and let $c_0$ be the number of fractional slack coordinates of $z$.
Then $0\le d_0\le t-1$. If $d_0=0$, the vertex $z$ is integral
and one may take $u=z$.
If $d_0>0$, put $k=d_0+\max\{0,c_0-1\}$. There is an integral
cost minimum $u\in F_q$ such that, for
$l=u-\lfloor z\rfloor$,
\begin{equation}\label{eq:matrix-proximity}
 \sum_j(l_{Y_j})^-\le k(Q-1),\qquad
 \sum_j(l_{Y_j})^+\le k(Q-1)+d_0.
\end{equation}
In either case, an integral cost minimum can be chosen with
\begin{equation}\label{eq:matrix-proximity-norm}
 \|u_Y-z_Y\|_1\le2mQ.
\end{equation}
\end{proposition}

\begin{proof}
Among the integral cost minima choose $u$ minimizing
$\|u-\lfloor z\rfloor\|_1$. Set
\[
 l=u-\lfloor z\rfloor,\quad f=z-\lfloor z\rfloor,\quad
 P=\sum_j(l_{Y_j})^+,\quad N=\sum_j(l_{Y_j})^-.
\]
There is no nonzero integer kernel vector $h\sqsubseteq l$.
Indeed,
\[
 h^+\le l^+\le u,\qquad
 h^-\le l^-\le\lfloor z\rfloor\le z,
\]
so $u-h$ and $z+h$ are respectively integral and real feasible
points. Their optimalities imply $\cost(h)\le0$ and
$\cost(h)\ge0$. Thus $\cost(h)=0$, but replacing $u$ by $u-h$
strictly decreases its distance to $\lfloor z\rfloor$.

The last row of $Dl=Df$ gives
\begin{equation}\label{eq:degree-residual}
 P-N=d_0=\sum_j f_{Y_j},\qquad 0\le d_0\le t-1.
\end{equation}
The sum on the right is an integer with $t$ terms in $[0,1)$.
If $d_0=0$, all $Y$-coordinates of $z$ are integral, and $Dz=q$
forces its slack coordinates to be integral. The choice of $u$
then gives $u=z$.

Assume $d_0>0$. Pair the $N$ negative $Y$-occurrences in $l$
with positive ones, leaving $d_0$ unpaired positive occurrences.
Take the paired differences and all signed slack occurrences as
the original multiset for Lemma~\ref{lem:cycles}. Every vector
has $\triangle$-norm at most $d$. A nonempty zero-sum submultiset
would restore to a nonzero integer kernel vector conformal to $l$;
the disjoint positive and negative coordinate supports prevent
cancellation. Such a submultiset is therefore excluded.

Let $S_0$ be the sum of the original multiset. The projected rows
of $Dl=Df$ give
\begin{equation}\label{eq:S0}
 S_0=\sum_jf_{Y_j}p_j
       -\sum_{\text{unpaired occurrences }j}p_j+f_X
     =U-V+f_X.
\end{equation}
Here $U,V\ge0$ and both coordinate sums are at most $d_0d$.
The negative coordinate sum of $S_0$ is at most $d_0d$.
If $c_0=0$, then $f_X=0$ and its positive coordinate sum is
also at most $d_0d$. If $c_0\ge1$, then
$\sum_i(f_X)_i<c_0$, so
\[
 \sum_i(S_0)_i^+<d_0d+c_0.
\]
The left-hand side is an integer. Hence
\begin{equation}\label{eq:triangle-residual}
 \sum_i(S_0)_i^-\le d_0d\le kd,\qquad
 \sum_i(S_0)_i^+\le d_0d+\max\{0,c_0-1\}\le kd.
\end{equation}

Lemma~\ref{lem:signedmass} writes $-S_0$ as a sum of $k$ integer
auxiliary vectors of $\triangle$-norm at most $d$.
Apply Lemma~\ref{lem:Steinitz} to the combined multiset.
All partial sums lie in the same set of $Q$ lattice points in
$\Lambda B_\triangle$. Lemma~\ref{lem:cycles} gives
$N\le k(Q-1)$, and $P=N+d_0$ proves
\eqref{eq:matrix-proximity}.

Finally, $t+c_0\le m+1$, because $z$ is a vertex and every
fractional slack coordinate is positive. Together with
$d_0\le t-1$, this gives $k\le m$. Since $d_0\le k$,
\[
 \|u_Y-z_Y\|_1\le\|l_Y\|_1+\|f_Y\|_1
 =P+N+d_0\le2k(Q-1)+2d_0\le2mQ.
\]
This proves~\eqref{eq:matrix-proximity-norm}.
\end{proof}

\section{Socle modules and stability of associated primes}\label{sec:algebra}

Return to a nonzero proper monomial ideal
$I\subset R=\kk[x_1,\ldots,x_r]$, where $r\ge2$.
Put $d=d(I)$ and $s=\mu(I)$ as in the introduction.
All associated primes are taken
over the indicated ring, and depth is computed at the homogeneous
maximal ideal. Associated primes of monomial ideals are generated
by variables and admit monomial colon witnesses; see
\cite[Chapter~1]{HH}.

For a nonempty subset $E\subseteq[r]$, set
\[
 R_E=\kk[x_i:i\in E],\qquad
 P_E=(x_i:i\in E)R,\qquad
 \mathfrak m_E=(x_i:i\in E)R_E.
\]
Let $I_E\subseteq R_E$ be obtained by setting $x_j=1$ for $j\notin E$.

\begin{lemma}\label{lem:local}
For every $n\ge1$,
\[
 P_E\in\Ass_R(R/I^n)
 \quad\Longleftrightarrow\quad
 \mathfrak m_E\in\Ass_{R_E}(R_E/I_E^n).
\]
If $I_E\ne R_E$, then $\mu(I_E)\le\mu(I)$ and $d(I_E)\le d(I)$.
\end{lemma}

\begin{proof}
Let $S$ be the multiplicative set generated by the variables
$x_j$ with $j\notin E$. Since specialization commutes with powers,
\[
 S^{-1}(R/I^n)\simeq
 (R_E/I_E^n)[x_j^{\pm1}:j\notin E].
\]
Associated primes commute with localization and extend under
Laurent polynomial extension. The extension of $\mathfrak m_E$
is $S^{-1}P_E$, whose contraction is $P_E$. This proves the
equivalence. The minimal generators of $I_E$ form a subset of the
projections of $G(I)$, which proves the remaining assertions.
\end{proof}

We shall also use the following fact; see~\cite[Remark, p.~47]{Trung}.

\begin{lemma}\label{lem:two}
Every nonzero proper monomial ideal in two variables has index of
stability one.
\end{lemma}

A nonzero principal monomial ideal also has index of stability one:
the associated primes of each positive power are precisely the
height-one monomial primes corresponding to variables dividing
its generator. The same observation covers the one-variable case.
Henceforth, in the technical arguments, assume $r\ge3$ and $s\ge2$.

\subsection{The Rees quotient and optimal integer solutions}

Write
\[
 G(I)=\{x^{a_1},\ldots,x^{a_s}\},\qquad
 a_j=(a_{1j},\ldots,a_{rj})\in\NN^r.
\]
Put $m=r-1$, $a'_j=(a_{2j},\ldots,a_{rj})$, and
$p_j=a'_j$ in the matrix~\eqref{eq:D}. Label its slack columns
$X_2,\ldots,X_r$, and use the linear cost
\[
 \cost(\kappa,\beta)=\sum_j a_{1j}\beta_j.
\]
Thus $D$ records the projected exponent and the power of $I$,
whereas the objective records the exponent of $x_1$.

Set
\[
 \Rees(I)=R[I\tau]=\bigoplus_{n\ge0}I^n\tau^n,\qquad
 C=\Rees(I)/x_1\Rees(I).
\]
Use the presentation
\[
 T=\kk[X_1,\ldots,X_r,Y_1,\ldots,Y_s]\longrightarrow\Rees(I),
 \qquad X_i\longmapsto x_i,\quad Y_j\longmapsto x^{a_j}\tau.
\]
The Rees grading is $\deg X_i=0$, $\deg Y_j=1$. We also use the
multigrading by the full exponent and the Rees degree. Every
multigraded component of $C$ has dimension at most one. Write
$X'=(X_2,\ldots,X_r)$, and let $L\subseteq T$ be the monomial
ideal generated by all monomials whose images in $C$ are zero.
The kernel of $T\to C$ may also contain binomials; we use $L$
only to test whether an individual monomial class vanishes.
\begin{lemma}[Integral optima and nonzero classes]\label{lem:survive}
Let $u=(\kappa,\beta)$. The class of $X'^\kappa Y^\beta$ in $C$
is nonzero if and only if $u$ minimizes cost over the integral
fiber $F_{Du}$. All integral optima in this fiber represent the
same nonzero multihomogeneous class.
\end{lemma}

\begin{proof}
The projected fiber fixes the projected exponents and the Rees degree
$n$; the exponent in the distinguished coordinate is $\cost(\kappa,\beta)$.
If a smaller integral cost exists, its difference from this cost is
at least one. Multiplying the corresponding Rees monomial by a
suitable nonnegative power of $x_1$ represents the original
monomial divided by $x_1$. Thus the original class vanishes in $C$.
Conversely, a representation after division by $x_1$, with any
$X_1$ factors omitted, gives a point of the same projected fiber
with smaller cost.
Integral minima have both the same cost and the same projected
exponents, hence the same full Rees exponent.
\end{proof}

\subsection{The socle module and its generating degrees}

Let
\[
 A=\kk[Y_1,\ldots,Y_s],\qquad A_+=(Y_1,\ldots,Y_s),\qquad
 M=0:_C\mathfrak m,
\]
where $\mathfrak m=(x_1,\ldots,x_r)$. The action of $A$ on $C$
is induced by the presentation above. The module $M$ is the
$\mathfrak m$-socle of $C$.
\begin{lemma}\label{lem:M}
For every $n\ge0$, the map
\[
 (I^n:\mm)/I^n\longrightarrow M_n,\qquad
 u+I^n\longmapsto[x_1u\tau^n],
\]
is an isomorphism of $\kk$-vector spaces.
It preserves the Rees degree $n$.
For $n\ge1$, therefore,
\begin{equation}\label{eq:socle-detection}
 M_n\ne0
 \quad\Longleftrightarrow\quad
 \mm\in\Ass_R(R/I^n).
\end{equation}
\end{lemma}

\begin{proof}
A class $[v\tau^n]\in C_n$ belongs to $M$ precisely when
$v\in I^n$ and $x_iv\in x_1I^n$ for every $i\ge2$.
The condition for $i=2$ implies $x_2v\in(x_1)$.
Since $(x_1):x_2=(x_1)$, it follows that $x_1\mid v$;
write $v=x_1u$.
These conditions are then exactly $x_iu\in I^n$ for
$1\le i\le r$.
The class vanishes if and only if $u\in I^n$, proving the
isomorphism. The final assertion is the usual socle criterion
for an associated maximal ideal.
\end{proof}

\begin{proposition}\label{prop:common}
Suppose that $L$ is generated in $Y$-degrees at most an integer
$b\ge1$.
\begin{enumerate}
\item The graded $A$-module $M$ is finitely generated in degrees at most
\[
 g(b)=\min\{(r-1)b,sb-1\}.
\]
\item If a nonzero multihomogeneous class of $C$ is annihilated by
every $Y_j$, then every monomial representation
$X'^\kappa Y^\beta$ of that class satisfies
\[
 \beta_j\le b-1\qquad(1\le j\le s).
\]
\end{enumerate}
\end{proposition}

\begin{proof}
For $i\ge2$, put $J_i=L:X_i$ and
$J_0=\bigcap_{i=2}^rJ_i$. The monomial-colon formula gives
generators $v/\gcd(v,X_i)$ for $J_i$, where $v$ ranges over the
minimal monomial generators of $L$. Their $Y$-degrees are at most
$b$. The intersection $J_0$ is generated by least common multiples
of one such generator from each $J_i$, so its minimal generators
have $Y$-degrees at most $(r-1)b$.

The module $M$ is multigraded. Since each multigraded component of
$C$ has dimension at most one, $M$ is spanned by monomial classes.
A nonzero monomial class $[u]$ belongs to $M$ precisely when
$u\in J_0\setminus L$, since $X_1$ already annihilates $C$.
Choose a minimal generator $v$ of $J_0$ dividing $u$. Then
$v\notin L$ and $[v]\in M$. The quotient $u/v$ contains no
$X_i$: otherwise $[u]=0$, since every $X_i$ annihilates $[v]$.
Thus $u/v\in A$, and the nonzero classes of the finitely many
generators of $J_0$ generate $M$ over $A$.

For the second generating-degree estimate, choose a minimal
multihomogeneous generating set of $M$ and represent one of its
elements by $v=X'^\kappa Y^\beta$. If $\beta_j>0$, then
$v/Y_j\notin L$ and its class does not belong to $M$, by
minimality. Hence some $i\ge2$ satisfies
\[
 X_iv\in L,\qquad X_iv/Y_j\notin L.
\]
A minimal generator of $L$ dividing $X_iv$ but not $X_iv/Y_j$
has $Y_j$-exponent $\beta_j$. Consequently $\beta_j\le b$.
Choose $j_0$ minimizing $a_{1j}$. No monomial
$X'^\alpha Y_{j_0}^a$ belongs to $L$, since its omitted-coordinate
exponent $a\,a_{1j_0}$ is the smallest possible cost among all
representations of Rees degree $a$. If $\beta_{j_0}>0$, the
minimal generator of $L$ just found must therefore involve another
$Y$-variable, and $\beta_{j_0}\le b-1$. This also holds when
$\beta_{j_0}=0$. Summing gives $\deg v\le sb-1$.

Finally, let $u=X'^\kappa Y^\beta\notin L$ represent a class
annihilated by every $Y_j$. A minimal generator of $L$ dividing
$uY_j$ but not $u$ has $Y_j$-exponent $\beta_j+1$. Its total
$Y$-degree is at most $b$, so $\beta_j\le b-1$.
\end{proof}

We record the elementary graded-module fact that will turn these
degree estimates into stabilization bounds.

\begin{lemma}\label{lem:graded-stability}
Let $N$ be a finitely generated nonnegatively graded module over
$A$, generated in degrees at most $g$.
\begin{enumerate}
\item If $N$ is infinite-dimensional over $\kk$, then $N_n\ne0$
for every $n\ge g$.
\item If $N$ is finite-dimensional and nonzero, its last nonzero
component is annihilated by $A_+$.
\end{enumerate}
Thus, if in the finite-dimensional case the degrees of
$0:_N A_+$ are at most $h$, then the condition $N_n\ne0$ is
independent of $n$ for $n\ge\max\{1,g,h+1\}$.
\end{lemma}

\begin{proof}
Since $A$ is standard graded, $N_{n+1}=A_1N_n$ for $n\ge g$.
A zero component in this range forces every subsequent component
to vanish. As all components are finite-dimensional, this would
make $N$ finite-dimensional. This proves the first assertion.
The second follows because multiplication by any $Y_j$ raises
degree. If $N$ is finite-dimensional and nonzero, its last nonzero
degree is therefore at most $h$. The final statement follows,
including the case $N=0$.
\end{proof}

\subsection{Compact faces and optimal supports}

The Newton polyhedron and analytic spread of $I$ are
\[
 \NP(I)=\conv\{a_1,\ldots,a_s\}+\RR_{\ge0}^r,\qquad
 \ell(I)=\dim\bigl(\Rees(I)/\mathfrak m\Rees(I)\bigr).
\]
The compact-face characterization of analytic spread was proved
by Bivi\`a-Ausina~\cite{Bivia}; for polynomial rings over arbitrary
fields, see Singla~\cite[Corollary~4.10]{Singla}. We give the
inequality and support argument needed here.
\begin{lemma}[Compact faces and analytic spread]\label{lem:compact}
For every nonempty compact face $F$ of $\NP(I)$,
\[
 \dim F+1\le\ell(I).
\]
\end{lemma}

\begin{proof}
Choose a supporting linear functional $c$ such that
\[
 F=\{x\in\NP(I):c\cdot x=\nu\},
 \qquad \nu=\min_{x\in\NP(I)}c\cdot x.
\]
All coordinates of $c$ are nonnegative, since every coordinate
vector is a recession direction of $\NP(I)$.  If $c_i=0$, then
$F+\RR_{\ge0}e_i\subseteq F$, contradicting compactness.
Thus $c_i>0$ for every $i$.  In particular,
\[
 \nu=\min_j c\cdot a_j,
 \qquad F=\conv\{a_j:a_j\in F\}.
\]

Write $B=\Rees(I)/\mm\Rees(I)$ for the fiber ring.
Its degree-$n$ component has a monomial basis indexed by
$G(I^n)$.  For the class represented by $x^b\tau^n$, define
its weight to be $c\cdot b-n\nu$.
Every such weight is nonnegative: a monomial in $I^n$ is divisible
by a product of $n$ generators of $I$.
Weights add whenever a product of basis classes is nonzero.
Consequently, the classes of positive weight span an ideal
$B_{>0}$ of $B$.

The quotient $B/B_{>0}$ is the affine semigroup ring
\[
 \kk[x^{a_j}\tau:a_j\in F].
\]
Indeed, the exponent $b$ of a product of $n$ generators on $F$
satisfies $c\cdot b=n\nu$. The product is a minimal generator of
$I^n$, since a proper divisor belonging to $I^n$ would have
$c$-weight strictly less than $n\nu$.
Conversely, any basis monomial of weight zero is a product of
generators on $F$: in a representation by $n$ generators and
nonnegative slack, weight zero excludes both positive slack
and generators outside $F$.
The semigroup ring in the display has dimension
\[
 \operatorname{rank}_{\ZZ}\langle(a_j,1):a_j\in F\rangle
 =\dim F+1.
\]
Since it is a quotient of $B$, this proves the assertion.
\end{proof}

\begin{lemma}[Support of an optimal solution]\label{lem:support}
For every nonempty real fiber $P_q$ of positive degree, a
cost-minimizing vertex can be chosen with at most $\ell(I)$
positive $Y$-coordinates.
\end{lemma}

\begin{proof}
Let $n\ge1$ be the prescribed degree.
The fiber $P_q$ is compact: its $Y$-coordinates are nonnegative
and sum to $n$, and its slack coordinates are determined by
these coordinates and the prescribed projected exponents.
It therefore has a cost-minimizing point.
Equivalently, let $w\in n\NP(I)$ have the prescribed
projected coordinates and the least possible first coordinate
$w_1$.  These minima agree because a point of $n\NP(I)$ is a
sum of a nonnegative combination of the $a_j$, with coefficients
summing to $n$, and nonnegative slack; positive slack in the
first coordinate can always be removed when minimizing it.

Put $P=n\NP(I)$ and
\[
 K=P\cap\prod_{i=1}^r[0,w_i].
\]
This is a nonempty compact polytope, since $w\in K$.
Choose $v\in K$ minimizing $\sum_i v_i$.
We claim that $v$ belongs to a compact face of $P$.
Write a finite halfspace representation
\[
 P=\{x\in\RR^r:c_h\cdot x\ge b_h\ (1\le h\le N)\}.
\]
Every $c_h$ is coordinatewise nonnegative, because $P$ is
closed under addition of nonnegative vectors.
For each coordinate $i$, at least one inequality active at $v$
has $(c_h)_i>0$.  Otherwise, the finitely many inactive
inequalities remain satisfied at $v-\varepsilon e_i$ for all
sufficiently small $\varepsilon>0$, and the active inequalities
remain equalities.  Thus $v-\varepsilon e_i\in P$.
This point still lies below $w$ and, being in $P$, is
nonnegative; hence it belongs to $K$, contradicting the choice
of $v$.

Let $c$ be the sum of the normals $c_h$ of all inequalities
active at $v$.  Then every coordinate of $c$ is strictly
positive, and $c\cdot x\ge c\cdot v$ for all $x\in P$.
Therefore
\[
 E=\{x\in P:c\cdot x=c\cdot v\}
\]
is a compact face containing $v$; compactness follows from
$c>0$ and $P\subseteq\RR_{\ge0}^r$.

We also have $v_1=w_1$.
Indeed, if $v_1<w_1$, upward closure and $v_i\le w_i$ for
$i\ge2$ imply
\[
 (v_1,w_2,\ldots,w_r)\in n\NP(I),
\]
contradicting the minimality of $w_1$.

The face $E/n$ of $\NP(I)$ is the convex hull of the generator
exponents that it contains; its strictly positive supporting
normal excludes every nonzero slack direction.
By Lemma~\ref{lem:compact}, $\dim E+1\le\ell(I)$.
Carath\'eodory's theorem in the affine hull of $E/n$ expresses
$v/n$ as a convex combination of at most $\ell(I)$ of those
generator exponents.
Multiplying the coefficients by $n$ and adding the projected
slack $w-v$ gives an optimal point of $P_q$ with at most
$\ell(I)$ positive $Y$-coordinates.
Set all other $Y$-coordinates equal to zero.  This defines a
nonempty face of $P_q$ containing an optimum.
An optimal vertex of that face is a vertex of $P_q$ and has
the required support bound.
\end{proof}

\begin{remark}\label{rem:finite-spread}
If $M$ is finite-dimensional, Lemma~\ref{lem:M} gives
$\depth R/I^n\ge1$ for all sufficiently large $n$. The asymptotic
Burch--Brodmann inequality~\cite{BrodmannSpread}
\[
 \ell(I)+\lim_{n\to\infty}\depth R/I^n\le r
\]
then implies $\ell(I)\le r-1$. This is the only place where
the depth of ordinary powers enters our estimates.
\end{remark}

\subsection{From Graver degrees to stabilization}

\begin{lemma}\label{prop:L}
If $\yd(h^+)\le b$ for every $h\in\mathcal G(D)$, where
$b\ge1$ is an integer, then $L$ is generated in $Y$-degrees at
most $b$.
\end{lemma}

\begin{proof}
The monomial $X_1$ belongs to $L$ and has $Y$-degree zero.
Let $v=X'^\kappa Y^\beta\in L$ have no $X_1$ factor, and put
$u=(\kappa,\beta)$. By Lemma~\ref{lem:survive}, some integral
point $u'$ satisfies $Du'=Du$ and $\cost(u')<\cost(u)$.
Decompose $u-u'$ conformally into Graver basis elements. At least
one summand $h$ has positive cost. The monomial corresponding to
$h^+$ divides $v$, belongs to $L$ by Lemma~\ref{lem:survive},
and has $Y$-degree at most $b$. Thus every monomial of $L$ has
a divisor in $L$ of $Y$-degree at most $b$.
\end{proof}

\begin{proof}[Proof of Theorem~\ref{thm:generators}]
Principal ideals and ideals in at most two variables were treated above. Assume $r\ge3$ and $s\ge2$.
Lemma~\ref{lem:rs-graver} and Lemma~\ref{prop:L} show that $L$
is generated in $Y$-degrees at most $\Gamma$. By
Proposition~\ref{prop:common}, $M$ is generated in degrees at most
\[
 g=\min\{(r-1)\Gamma,s\Gamma-1\}.
\]
If $M$ is infinite-dimensional, Lemma~\ref{lem:graded-stability}
gives $M_n\ne0$ for every $n\ge g$.

Suppose $M$ is finite-dimensional and nonzero, and choose a nonzero
multihomogeneous class $\eta$ in its last nonzero degree $n$.
It is annihilated by every $Y_j$. Proposition~\ref{prop:common}
gives $\beta_j\le\Gamma-1$ in every monomial representation of
$\eta$, and hence $n\le s(\Gamma-1)$.
If $n\ge1$, choose an optimal vertex $z\in P_q$ in its projected
multidegree with $t=|\supp(z_Y)|\le\ell(I)$ by
Lemma~\ref{lem:support}. An integral optimum as in
Lemma~\ref{lem:rs-proximity} represents the same class $\eta$.
There are at most $t(\Gamma-1)$ generator occurrences on the
support of $z_Y$ and at most $\Gamma-1$ outside it. Therefore
\[
 n\le(\ell(I)+1)(\Gamma-1).
\]
Remark~\ref{rem:finite-spread} gives $\ell(I)\le r-1$, so
\[
 \operatorname{end}M\le\min\{r,s\}(\Gamma-1).
\]
The same bound holds when $n=0$; the case $M=0$ needs no bound.
By Lemmas~\ref{lem:graded-stability} and~\ref{lem:M}, membership
of $\mathfrak m$ in $\Ass_R(R/I^n)$ is constant for
\begin{equation}\label{eq:rs-threshold}
 n\ge T(r,s,\Gamma):=
 \max\bigl\{\min\{(r-1)\Gamma,s\Gamma-1\},
                  \min\{r,s\}(\Gamma-1)+1\bigr\}.
\end{equation}

Apply the same argument to every monomial localization $I_E$.
If it has $q\ge3$ variables and $s_E\ge2$ minimal generators,
its parameters satisfy
\[
 q\le r,\quad s_E\le s,\quad
 (s_E-1)d(I_E)^{\min\{q-1,s_E-2\}}\le\Gamma.
\]
The function $T$ is nondecreasing in these parameters.
Localizations in at most two variables, or with one generator,
are stable from the first power; a localization equal to the
whole ring contributes no associated prime. Lemma~\ref{lem:local}
therefore gives a common threshold $T(r,s,\Gamma)$ for all primes.
Finally,
\[
 T(r,s,\Gamma)\le\min\{r,s\}\Gamma.
\]
This proves~\eqref{eq:generator-main}. If $s=2$, then $\Gamma=1$
and~\eqref{eq:rs-threshold} gives $T(r,2,1)=1$.
\end{proof}

\section{Lattice points and generator-independent stabilization}
\label{sec:lattice}

Keep the notation of Section~\ref{sec:algebra}, and assume
$r\ge3$ and $s\ge2$. In the matrix notation of
Section~\ref{sec:matrix}, we have $m=r-1$, $p_j=a'_j$, and
\[
 \Lambda=\lfloor\sigma_md\rfloor,\qquad Q=\Phi_m(\Lambda).
\]
Lemma~\ref{lem:primitive} bounds the Graver degrees by $Q$.
To apply this bound to stabilization, we control the last
nonzero degree of the socle module when it is finite-dimensional.
\subsection{Proximity and the last nonzero socle degree}

We apply Proposition~\ref{prop:matrix-proximity}, which includes
the contribution of fractional slack coordinates. The support
estimate from the analytic spread then improves the degree bound
in the finite-dimensional case.

\begin{proposition}\label{prop:triangle-top}
If $M$ is finite-dimensional and nonzero, then
\begin{equation}\label{eq:triangle-end}
 \operatorname{end}M\le(2r-3)Q-r+1.
\end{equation}
\end{proposition}

\begin{proof}
By Lemmas~\ref{lem:primitive} and~\ref{prop:L}, the ideal $L$
is generated in $Y$-degrees at most $Q$. Choose a nonzero
multihomogeneous monomial class $\eta$ in the last nonzero degree
$n$ of $M$. It is annihilated by every $Y_j$. Thus, in every
monomial representation $X'^\kappa Y^\beta$ of $\eta$,
\begin{equation}\label{eq:lattice-caps}
 \beta_j\le Q-1\qquad(1\le j\le s)
\end{equation}
by Proposition~\ref{prop:common}. The case $n=0$ is immediate,
so assume $n\ge1$, and let $q$ be its projected multidegree.
Choose an optimal vertex $z\in P_q$ with
$t=|\supp(z_Y)|\le\ell(I)$ by Lemma~\ref{lem:support}.
Since $M$ is finite-dimensional,
$t\le\ell(I)\le r-1$ by Remark~\ref{rem:finite-spread}.

Use the integral optimum $u$ of
Proposition~\ref{prop:matrix-proximity}, which represents $\eta$
by Lemma~\ref{lem:survive}. Retain $d_0,c_0,k$ from that
proposition. If $d_0=0$, then $u=z$ and
$n\le t(Q-1)$, giving~\eqref{eq:triangle-end}.
Suppose $d_0>0$. Its generator occurrences on $\supp(z_Y)$
number at most $t(Q-1)$ by~\eqref{eq:lattice-caps}.
Every occurrence outside this support contributes to
$(u_Y-\lfloor z_Y\rfloor)^+$. Therefore
\begin{equation}\label{eq:triangle-top}
 n\le t(Q-1)+k(Q-1)+d_0.
\end{equation}

We have $t+c_0\le r$ and $d_0\le t-1\le r-2$.
If $c_0=0$, then $t+k=t+d_0\le2t-1\le2r-3$.
If $c_0\ge1$, then
\[
 t+k=t+d_0+c_0-1\le2t+c_0-2\le r+t-2\le2r-3.
\]
Substitution in~\eqref{eq:triangle-top} gives
\[
 n\le(2r-3)(Q-1)+(r-2)=(2r-3)Q-r+1.
\]
\end{proof}

\begin{proposition}\label{prop:refined}
Let $r\ge3$, $m=r-1$, and $\Lambda=\lfloor\sigma_md\rfloor$.
Then
\begin{equation}\label{eq:refined}
 \astab(I)\le(2r-3)\Phi_m(\Lambda)-r+2.
\end{equation}
\end{proposition}

\begin{proof}
If $s=1$, the assertion follows from $\astab(I)=1$, since the
right-hand side of~\eqref{eq:refined} is at least one. Assume
$s\ge2$ and put $Q=\Phi_m(\Lambda)$. Lemmas~\ref{lem:primitive} and
\ref{prop:L}, followed by Proposition~\ref{prop:common}, show
that $M$ is generated in degrees at most $g=(r-1)Q$.
Proposition~\ref{prop:triangle-top} bounds its last nonzero degree
in the finite-dimensional case. Since
\[
 [(2r-3)Q-r+2]-g=(r-2)(Q-1)\ge0,
\]
Lemmas~\ref{lem:graded-stability} and~\ref{lem:M} imply that
membership of $\mathfrak m$ is constant for
$n\ge(2r-3)Q-r+2$.

For any monomial prime $P_E$, put $q=|E|$. If $I_E=R_E$, it
never occurs. If $q\le2$ or $I_E$ is principal, it is stable
from the first power. Otherwise, apply the preceding argument to
$I_E$. Since $d(I_E)\le d$ and
$\sigma_{q-1}\le\sigma_m$, extension by zero coordinates embeds
\[
 \lfloor\sigma_{q-1}d(I_E)\rfloor
 (\Delta_{q-1}-\Delta_{q-1})\cap\ZZ^{q-1}
\]
in $\Lambda B_\triangle\cap\ZZ^m$.
Thus membership of $P_E$ is constant from a threshold no larger
than $(2q-3)Q-q+2$. For $q\le r$ this is at most
$(2r-3)Q-r+2$, since the difference equals
$(r-q)(2Q-1)\ge0$. Lemma~\ref{lem:local} completes the proof.
\end{proof}

\section{The uniform coefficient}\label{sec:uniform}

We estimate the lattice-point bound and then combine it with
Theorem~\ref{thm:generators}. The map
\[
 (v_1,\ldots,v_m)\longmapsto
 (v_1,\ldots,v_m,-v_1-\cdots-v_m)
\]
identifies the lattice of $B_\triangle$ with the root lattice of
type $A_m$ and sends its nonzero vertices to the roots $e_i-e_j$.
Its Ehrhart series is therefore
\begin{equation}\label{eq:ehrhart-series}
 \sum_{L\ge0}\Phi_m(L)z^L
 =\frac{\sum_{j=0}^m\binom mj^2z^j}{(1-z)^{m+1}};
\end{equation}
see~\cite[Example~3.5]{HKM}. We also write $\Phi_m(x)$ for its
Ehrhart polynomial evaluated at a real variable.

\begin{lemma}\label{lem:ehrhart-positive}
For $m\ge2$,
\begin{equation}\label{eq:ehrhart-recurrence}
 m^2\Phi_m(x)=(2m-1)(2x+1)\Phi_{m-1}(x)
                  +(m-1)^2\Phi_{m-2}(x),
\end{equation}
where $\Phi_0(x)=1$ and $\Phi_1(x)=2x+1$.
Every coefficient of $\Phi_m(x)$ and of
\[
 P_m(x)=(2m-1)\Phi_m(x)-m+1
\]
in the power basis is positive. In particular, for every $a>0$,
the function $P_m(ax)/x^m$ is nonincreasing on $x>0$.
\end{lemma}

\begin{proof}
The recurrence is given in~\cite[Example~3.5]{HKM} and also
follows from~\eqref{eq:ehrhart-series}. It implies coefficient
positivity by induction. Since $\Phi_m(0)=1$, the constant
coefficient of $P_m$ is $m$, and all other coefficients are
positive. If $P_m(x)=\sum_{j=0}^m b_jx^j$, then
\[
 \frac{P_m(ax)}{x^m}=\sum_{j=0}^m b_ja^jx^{j-m},
\]
and every summand is nonincreasing.
\end{proof}

\begin{lemma}\label{lem:volume-count}
For every integer $L\ge0$,
\begin{equation}\label{eq:cube-count}
 \Phi_m(L)\le
 \frac1{m!}\binom{2m}{m}(L+m/2)^m.
\end{equation}
Consequently, for $d\ge2$,
\begin{equation}\label{eq:volume-coefficient}
 \frac{\astab(I)}{d^m}
 \le (2m-1)\frac{\binom{2m}{m}}{m!}
            \left(\sigma_m+\frac{m}{2d}\right)^m.
\end{equation}
\end{lemma}

\begin{proof}
The leading terms of~\eqref{eq:Phi} give
\[
 \operatorname{vol}(B_\triangle)
 =m!\sum_{p=0}^m\frac1{p!^2(m-p)!^2}
 =\frac1{m!}\binom{2m}{m}.
\]
Place the half-open unit cube $u+[-1/2,1/2)^m$ at each lattice
point $u\in LB_\triangle$. The cubes have disjoint interiors and
total volume $\Phi_m(L)$. Their union is contained in
$(L+m/2)B_\triangle$, since adding a vector in
$[-1/2,1/2]^m$ increases either sign sum by at most $m/2$.
This proves~\eqref{eq:cube-count}. Apply it with
$L=\lfloor\sigma_md\rfloor$ in Proposition~\ref{prop:refined}
and discard its negative final term to obtain
\eqref{eq:volume-coefficient}.
\end{proof}

\begin{proof}[Proof of Theorem~\ref{thm:main}]
The case $r=2$ is Lemma~\ref{lem:two}. Let $r\ge3$ and $m=r-1$.
Principal ideals have index of stability one, so assume $s\ge2$.
If $d=1$, then $I$ is generated by variables and $\astab(I)=1$.
Assume $d\ge2$.

First suppose $2\le m\le7$. Proposition~\ref{prop:refined} and
Lemma~\ref{lem:ehrhart-positive} give
\[
 \frac{\astab(I)}{d^m}
 \le\frac{P_m(\sigma_md)}{d^m}
 \le\frac{P_m(2\sigma_m)}{2^m}.
\]
For $2\le m\le7$, set
\[
 R_m=\frac{4P_m(2\sigma_m)}{2^m(54/5)^m}.
\]
Exact rational substitution in~\eqref{eq:ehrhart-recurrence} gives
the following componentwise inequalities:
\[
 (R_2,R_3,R_4,R_5,R_6,R_7)
 <\frac1{1000}(944,924,918,920,926,934).
\]
Thus $P_m(2\sigma_m)/2^m<(54/5)^m/4<(4e)^m/4$, where
$e>27/10$ was used in the last inequality.

Now suppose $m\ge8$ and $2\le d<2m$. A minimal monomial
generating set of degree at most $d$ has at most
$\binom{d+m}{m}$ elements. Indeed, multiply each generator by a
power of a fixed variable to make its degree equal to $d$.
Equal resulting monomials would make the original generators
comparable under divisibility, so this map is injective.
Theorem~\ref{thm:generators} gives
\begin{align*}
 \frac{\astab(I)}{d^m}
 &\le(m+1)(s-1)
 \le(m+1)\left[\binom{d+m}{m}-1\right]\\
 &\le(m+1)\binom{3m}{m}
 \le(m+1)(27/4)^m.
\end{align*}
The last inequality follows from the term of index $m$ in
$(1/3+2/3)^{3m}=1$. Dividing by $(4e)^m$ gives a quantity
smaller than $(m+1)(5/8)^m$. This expression decreases for
$m\ge8$, and $9(5/8)^8<1/4$.

It remains to consider $m\ge8$ and $d\ge2m$. The standard
Wallis--Stirling inequalities
\[
 \binom{2m}{m}\le\frac{4^m}{\sqrt{\pi m}},\qquad
 m!\ge\sqrt{2\pi m}\left(\frac me\right)^m
\]
and Lemma~\ref{lem:volume-count} yield
\begin{align*}
 \frac{\astab(I)}{d^m}
 &\le(2m-1)\frac{\binom{2m}{m}}{m!}
                     \left(\sigma_m+\frac14\right)^m\\
 &\le\frac{\sqrt2}{\pi}
       \left(1-\frac{3}{4m}+\frac1{m^2}\right)^m(4e)^m\\
 &\le\frac{\sqrt2}{\pi}e^{-5/8}(4e)^m
 <\frac14(4e)^m.
\end{align*}
Here $m\ge8$ gives $-3/4+1/m\le-5/8$. For the last inequality,
use $\sqrt2/\pi<11/24$ and
\[
 e^{5/8}>1+\frac58+\frac{(5/8)^2}{2}
                         +\frac{(5/8)^3}{6}>\frac{11}{6}.
\]
These cases prove~\eqref{eq:main}.
\end{proof}

\section*{Acknowledgements}
The author is partially supported by the Research Project NAFOSTED 101.04-2024.07.

\end{document}